\documentclass[11pt]{amsart}

 \usepackage[margin=1in]{geometry}

\usepackage{amssymb, amsmath, amscd, amsthm, color, epsfig,url, graphicx}

\usepackage[all]{xy}          
\xyoption{dvips}              

\makeatletter
\renewcommand\l@subsection{\@tocline{2}{0pt}{2pc}{5pc}{}}
\makeatother

\newcommand{\R}{{\mathbb R}}

\newcommand{\hofiber}{\operatorname{hofiber}}

\newcommand{\holim}{\operatorname{holim}}

\newcommand{\tfiber}{\operatorname{tfiber}}

\newcommand{\Emb}{\operatorname{Emb}}

\newcommand{\Imm}{\operatorname{Imm}}

\newcommand{\rImm}{\operatorname{rImm}}

\newcommand{\Ho}{\operatorname{H}}

\newcommand{\conf}{\operatorname{Conf}}
\newcommand{\Conf}{\operatorname{Conf}}
\newcommand{\rConf}{\operatorname{rConf}}

\newcommand{\calX}{{\mathcal{X}}}
\newcommand{\calY}{{\mathcal{Y}}}
\newcommand{\calZ}{{\mathcal{Z}}}
\newcommand{\calC}{{\mathcal{C}}}
\newcommand{\calO}{{\mathcal{O}}}
\newcommand{\calP}{{\mathcal{P}}}

\newcommand{\Top}{\operatorname{Top}}

\theoremstyle{plain}
\newtheorem{thm}{Theorem}[section]
\newtheorem{prop}[thm]{Proposition}
\newtheorem{lemma}[thm]{Lemma}
\newtheorem{cor}[thm]{Corollary}

\theoremstyle{definition}
\newtheorem{defin}[thm]{Definition}
\newtheorem{example}[thm]{Example}
\newtheorem{def/ex}[thm]{Definition/Example}

\theoremstyle{remark}
\newtheorem{rem}[thm]{Remark}

\newcommand{\refS}[1]{Section~\ref{S:#1}}
\newcommand{\refT}[1]{Theorem~\ref{T:#1}}
\newcommand{\refC}[1]{Corollary~\ref{C:#1}}
\newcommand{\refP}[1]{Proposition~\ref{P:#1}}
\newcommand{\refD}[1]{Definition~\ref{D:#1}}
\newcommand{\refL}[1]{Lemma~\ref{L:#1}}
\newcommand{\refR}[1]{Remark~\ref{R:#1}}

\newcommand{\refEx}[1]{Example~\ref{Ex:#1}}

\begin{document}


\title[]{Low stages of the Taylor tower for $r$-immersions}


\author{Bridget Schreiner}
\address{Department of Mathematics, University of Notre Dame, Notre Dame, IN}
\email{bschrein@nd.edu}

\author{Franjo \v Sar\v cevi\'c}
\address{Department of Mathematics, University of Sarajevo, Bosnia and Herzegovina}
\email{franjo.sarcevic@live.de}

\author{Ismar Voli\'c}
\address{Department of Mathematics, Wellesley College, Wellesley, MA}
\email{ivolic@wellesley.edu}
\urladdr{ivolic.wellesley.edu}

%


\begin{abstract}
We study the beginning of the Taylor tower, supplied by manifold calculus of functors, for the space of $r$-immersions, which are immersions without $r$-fold self-intersections.  We describe the first $r$ layers of the tower and discuss the connectivities of the associated maps.   We also prove several results about $r$-immersions that are of independent interest.
\end{abstract}
\maketitle
\tableofcontents

\parskip=4pt
\parindent=0cm


\section{Introduction}\label{S:Intro}


Let $M$ and $N$ be smooth manifolds.  An \emph{$r$-immersion} of $M$ in $N$ is an immersion that has no $r$-fold self-intersections, i.e.~no $r$ points of $M$ are mapped to the same point in $N$.  The purpose of this paper is to initiate the study of  $\rImm(M,N)$, the space of $r$-immersions of $M$ in $N$ (see \refD{Spaces}), using manifold calculus of functors.  

This theory, due to Goodwillie and Weiss \cite{GW:EI2, W:EI1} (see also \cite{M:MfldCalc} and \cite[Section 10.2]{MV:Cubes} for overviews), studies contravariant functors $F\colon \calO(M)\to \calC$ where $\calO(M)$ is the poset of open subsets of $M$ and $\calC$ is usually $\Top$, the category of topological spaces (but it can be the category of spectra, chain complexes, etc.).  The theory produces a \emph{Taylor tower} of functors and natural transformations
$$
F(-)\longrightarrow\big(T_\infty F(-)\to\cdots\to T_kF(-)\to T_{k-1}F(-)\to\cdots \to T_0F(-)  \big)
$$
whose stages $T_kF(-)$ approximate $F$ in a suitable ``polynomial'' sense (see \refS{ManifoldCalculus}).  The hope is that the Taylor tower converges to $F$, namely that there is an equivalence between $F$ and $T_\infty F$, the inverse limit of the tower.

The main functor to which manifold calculus has been applied with great success is the embedding functor (in fact, this functor is the motivation for the development of the theory).  Namely, the space of embeddings $\Emb(M,N)$ can be regarded as a contravariant functor on $\calO(M)$ since an inclusion of open subsets of $M$ gives a restriction map of embedding spaces.  An important and deep result \cite{GK, GW:EI2} states that 
\begin{equation}\label{E:EmbConnectivities}
\text{the map \ $\Emb(M,N)\to T_k\Emb(M,N)$\ \   is\ \  $(k(n-m-2)-m+1)$-connected},
\end{equation}
where $m$ and $n$ are the dimensions of $M$ and $N$, respectively.
(This is restated later as \refT{EmbeddingsConvergence} and more details are given there.)  Thus if $n-m-2>0$, the Taylor tower converges to $\Emb(M,N)$. 

This convergence result has proven to be remarkably fruitful in the study of knot and link spaces \cite{ALTV, DwyerHess:LongKnots, LTV:Vass, MV:Links, S:OKS, PAST:LinkSSCollapse, V:FTK}, embeddings of long planes \cite{AT:LongPlanes2, BW:EmbConfCat}, more general embedding spaces \cite{ALV, M:Emb}, link maps and their Milnor invariants \cite{GM:LinksEstimates, M:LinkNumber, M:Milnor}, etc. A strong connection to operads has also been established and this point of view pervades much of the current work in manifold calculus.  Because of its success with embeddings, the foundation of the theory has also been expanded in various directions \cite{BW:CalculusSheaves, MV:Multi, T:MfldCalcComplex, T:Context-Free}.

When $r=2$, the space of $r$-immersions is precisely the space of embeddings since the condition is that such an immersion may not have double points, and this defines an embedding (for $M$ compact). Thus  $\Emb(M,N)$ is just the beginning of the filtration
$$
\Emb(M,N)=2\Imm(M,N)\subset3\Imm(M,N)\subset\cdots\subset \rImm(M,N)\subset\cdots\subset \Imm(M,N).
$$
The last space is that of immersions of $M$ in $N$.  All the spaces above are contravariant functors on $\calO(M)$ and it is thus natural to try to extend the study of embeddings to $r$-immersions using manifold calculus.  The beginning of such a study, from the operad point of view, already appears in \cite[Section 11]{DT:Overlapping}, although that work does not address the convergence question, which is our main concern here.

Another motivation for developing manifold calculus for $r$-immersions is recent developments in combinatorial topology.  One of the goals in this field is to understand and control the self-intersections of maps $K\to\R^n$ where $K$ is an $m$-dimensional complex.  One of the best-known of this family of questions is the \emph{Tverberg Conjecture}:   For $r\geq 2$  and $n\geq 1$, any map $f\colon \Delta^{(r-1)(n+1)}\to \R^n$ maps points from $r$ disjoint faces to the same point.  This conjecture has been disproved by Frick \cite{Frick:Tverberg} who uses work of Mabillard and Wagner \cite{MW:TverbergI, MW:TverbergII} on the generalized Whitney trick that gives a way of resolving self-intersection points and formulates obstructions for doing this.

On the other hand, the Taylor tower for embeddings has been shown to capture and classify obstructions for resolving self-intersections \cite{GKW:EmbDisjSurg, M:Emb} and for turning an immersion into an embedding. The corresponding question of turning an immersion into an $r$-immersion would then have a home in the Taylor tower for $r$-immersions, 
as consequently the Tverberg Conjecture might naturally live there as well.  Futhermore, the machinery of manifold calculus might provide new context for many related ``Tverberg--vanKampen--Flores-type'' questions.


However, the tower studied here is still one step removed from Tverberg-type problems due to the condition in such problems that the self intersections should come from ``far away''.  Therefore, what is needed after the Taylor tower $r$-immersions is understood is a development of the tower for the subspace of $r$-immersions of a manifold, now with a triangulation, given by those self-intersections that come from disjoint simplices (these are sometimes in the literature called \emph{almost} $r$-immersions).  A recent extension of manifold calculus to simplicial complexes \cite{T:MfldCalcComplex} should be relevant here.

We also expect that the Taylor tower for $r$-immersions will connect to work of Salikhov \cite{S:MultipleImmersions} on obstructions to the existence of a homotopy from an immersion to an $r$-immersion. 

\subsection{Main results and organization}

The results of this paper are summarized in the following diagram.

{\footnotesize
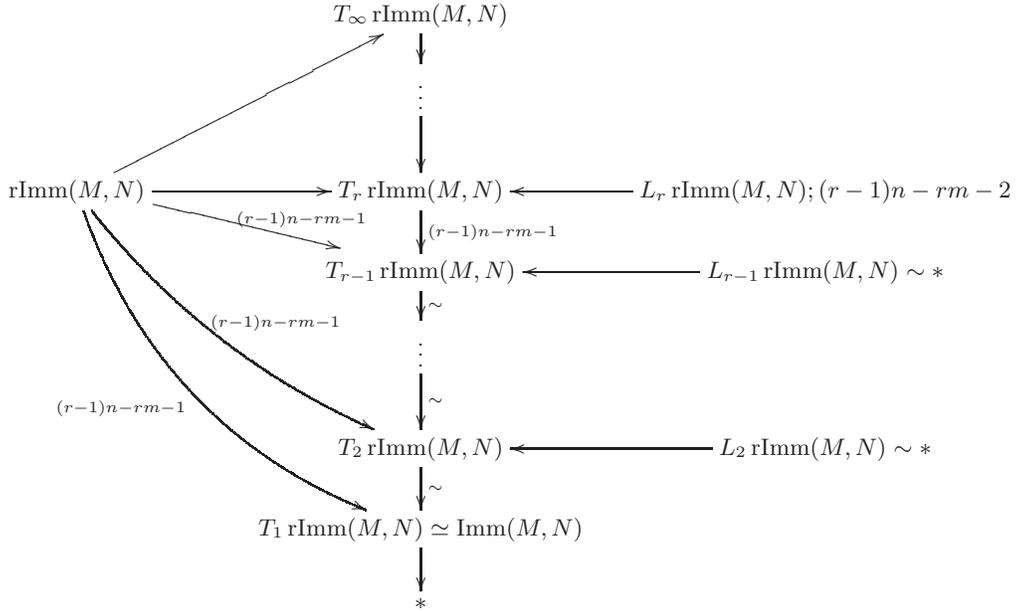
\begin{equation}\label{E:MainDiagram}
\xymatrix@=16pt{
       &  &  T_{\infty}\rImm(M,N)\ar[d] & \\
       &  &   \vdots \ar[d]  \\
 \rImm(M,N)\ar[uurr]\ar[rr]\ar@<0.5ex>[drr]^(0.6){(r-1)n-rm-1}\ar@/_1pc/[dddrr]^-{(r-1)n-rm-1} \ar@/_2pc/[ddddrr]_-{(r-1)n-rm-1}  &  &    T_r\rImm(M,N)\ar[d]^-{(r-1)n-rm-1}  & L_r\rImm(M,N); \text{\footnotesize $(r-1)n-rm-2$}  \ar[l]\\
       &  &    T_{r-1}\rImm(M,N) \ar[d]^-\sim & L_{r-1}\rImm(M,N)\sim * \ar[l]\\
            &  &   \vdots \ar[d]^-\sim & \\
       &  &     T_{2}\rImm(M,N)\ar[d]^-\sim & L_{2}\rImm(M,N)\sim * \ar[l] \\
                &  &   T_1\rImm(M,N)\simeq\Imm(M,N)\ar[d] & \\
                && \ast &
}
\end{equation}
}

The spaces $L_k\rImm(M,N)$ are the (homotopy) fibers of the maps $T_k\rImm(M,N)\to T_{k-1}\rImm(M,N)$.  The numbers in the diagram are connectivities of maps and, in the case of $L_r\rImm(M,N)$, the connectivity of that space.  The first stage of the Taylor tower is a one-point space by definition (see \refD{TaylorStage}).  The equivalence between $T_1\rImm(M,N)$ and $\Imm(M,N)$ is the content of \refP{T_1rImm=Imm}.  \refP{L_krImmContractible} then shows that $L_k\rImm(M,N)$ is contractible for $2\leq k\leq r-1$, from which it follows (\refT{T_kSameFork<r}) that there are equivalences  
$T_k\rImm(M,N)\to T_{k-1}\rImm(M,N)$ for  $2\leq k\leq r-1$.  We also show in \refC{rImm->ImmConnectivity}  that the map $\rImm(M,N)\to \Imm(M,N)$ is $((r-1)n-rm-1)$-connected, and so all the connectivities of the maps $\rImm(M,N)\to T_k\rImm(M,N)$, $2\leq k\leq r-1$, follow.  Lastly, in \refP{L_rrImmConnectivity} we exhibit the connectivity of $L_r\rImm(M,N)$.

There are two main lines of arguments we employ in our proofs.  One is general position and transversality, which are well-known topics and we use them in a basic way.  The other ingredient is homotopy limits of cubical diagrams; these techniques are central in functor calculus, but are less known so we review them in \refS{Cubes}.  

We also provide a review of manifold calculus and the way it applies to the embedding functor in \refS{ManifoldCalculus}.  This section is of independent interest since it provides a short path through the theory while supplying many references for further reading.

Also of independent interest are several results that have to do with spaces of $r$-immersions and $r$-configuration spaces, such as  \refT{rImm->(r+1)ImmConnectivity},  \refP{r+1R^n->rR^n}, and \refL{N-diagonal->N}.  We hope that some readers will find these useful regardless of their interest in calculus of functors.

Mentioned above are $r$-configuration spaces, which are central to our story.  These are  configuration spaces where up to $r-1$ points are allowed to be the same.  We will say more about them in \refS{Configurations} but it is worth noting that these are generally difficult and worthy of investigation in their own right.  In fact, the reason why our results stop at the $r^{th}$ stage of the Taylor tower is that this is the range in which $r$-configuration spaces are very simple.  

Some ideas and future directions of investigation of the higher stages of the Taylor tower for $r$-immersions are given in \refS{HigherStages}.





\subsection*{Acknowledgements} The third author would like to thank Tom Goodwillie, Sadok Kallel, and Rade \v Zivaljevi\'c  for helpful conversations, as well as the Simons Foundation for its support.



\section{Spaces of embeddings, immersions, and $r$-immersions}\label{S:Spaces}


Suppose $M$ and $N$ are smooth manifolds of dimensions $m$ and $n$, respectively.  We will assume throughout that $m\leq n$.  Let $\calC^\infty(M,N)$ be the space of smooth maps from $M$ to $N$, topologized using the Whitney $\calC^\infty$ topology.  The spaces in the following definition are all topologized as subspaces of $\calC^\infty(M,N)$.

\begin{defin}\label{D:Spaces}  \ 
\vspace{-6pt}
\begin{itemize}
\item An \emph{embedding} of $M$ in $N$ is a smooth map $f\colon M\to N$ satisfying 
\begin{enumerate}
\item[(i)] $f$ is a homeomorphism onto its image, and
\item[(ii)] the derivative of $f$ is injective, i.e.~the map of tangent spaces $D_xf\colon T_xM\to T_{f(x)}N$ is an injection for all $x\in M$.
\end{enumerate}
The space of embeddings of $M$ in $N$ is denoted by $\Emb(M,N)$.  A path in the space of embeddings is called an \emph{isotopy}.
\vspace{4pt}
\item An \emph{immersion} of $M$ in $N$ is a smooth map $M\to N$ satisfying (ii) above. The space of immersions of  $M$ in $N$ is denoted by $\Imm(M,N)$.  A path in the space of immersions is called a \emph{regular homotopy}.
\vspace{4pt}
\item An \emph{$r$-immersion} of $M$ in $N$ is an immersion of $M$ in $N$ that does not have $r$-fold intersections, i.e.~it satisfies the property that, for any subset of $r$ distinct points $R=\{x_1, ..., x_r\}$ of $M$, it is not a constant map when restricted to $R$.
The space of $r$-immersions of $M$ in $N$ is denoted by $\rImm(M,N)$. 
\end{itemize}
\end{defin}

For a compact manifold $M$, an embedding is an injective immersion.
It is then immediate from the above definition that, if $M$ is compact,  
$$
2\Imm(M,N)=\Emb(M,N),
$$
because the condition that an immersion have no double points is the same as requiring it to be injective.  
Much of the time, $M$ will for us indeed be a compact manifold and we will indicate when this assumption is being made.

Here is a result that is of independent interest.  Its consequence, \refC{rImm->ImmConnectivity}, will turn out to say something about the beginning of the Taylor tower for $r$-immersions (\refC{Connectivityk<r}).  The reader unfamiliar with the definition of the connectivity of a map should look ahead at \refD{k-connectedSpace}.

\begin{thm}\label{T:rImm->(r+1)ImmConnectivity}
Let $M$ and $N$ be smooth manifolds of dimensions $m$ and $n$, respectively.  Then the inclusion
$$
\rImm(M,N) \longrightarrow (r+1)\!\Imm(M,N)
$$
is $((r-1)n-rm-1)$-connected.
\end{thm}


\begin{proof}
The proof is a standard general position argument.
Namely, consider a map
\begin{align*}
h\colon  S^k & \longrightarrow (r+1)\!\Imm(M,N) \\
 v & \longmapsto f
\end{align*}
We wish to find for which $k$ there exists a lift
\begin{equation}\label{E:LiftTorImm}
\xymatrix{
 & \rImm(M,N)\ar[d]\\
S^k 
\ar@{-->}[ur]
\ar[r]^-h & (r+1)\!\Imm(M,N)
}
\end{equation}
Consider the adjoint $H$ of $h$, i.e.
\begin{align*}
H\colon  M\times S^k & \longrightarrow N \\
 (x,v) & \longmapsto f(x)
\end{align*}
and the map

\begin{align}
\widetilde H\colon  \Conf(r,M)\times S^k & \longrightarrow N^r \label{E:AssociatedMap}\\
 (x_1, x_2, ..., x_r, v) & \longmapsto (f(x_1), f(x_2), ..., f(x_r)). \notag
\end{align}
Here $\Conf(r,M)$ is the configuration space of $r$ distinct point in $M$ (see \eqref{E:ConfSpace} for the precise definition).

Now we want to use the multi-jet transversality theorem (see, for example, \cite[A.2.23]{MV:Cubes}). We are looking at zeroth order jets, in which case the theorem says the following: Suppose that $Z$ is a submanifold of $\Conf(r,M)\times N^r$. Then any $f\colon M\to N$ can be approximated by a map (i.e.~it is homotopic to a map by a small homotopy) such that the associated map 
\begin{align*}
\Conf(r,M) & \longrightarrow \Conf(r,M)\times N^r \\
(x_1,...,x_r) & \longmapsto  (x_1,...,x_r,f(x_1),...,f(x_r))
\end{align*}
is transverse to $Z$. We are interested in this in the case when $Z$ is the product of $\Conf(r,M)$ and the (thin) diagonal $\Delta$ in $N^r$,  i.e.~the set of $r$-tuples $(x_1,...,x_r)\in N^r$ such that $x_1=x_2=\cdots=x_r$. The transversality can then also be expressed by saying that the map 
\begin{align*}
\Conf(r,M) & \longrightarrow N^r \\
(x_1,..., x_r) & \longmapsto  (f(x_1),...,f(x_r))
\end{align*}
is transverse to the diagonal. 

What we really need is  a statement about parametrized families of maps, i.e.~maps $M\times P\to N$ where the parameter manifold $P$ is for us a sphere $S^k$.  This can be done by making the following choice: Define a submanifold $Z$ of $\Conf(r,M\times P)\times N^r$ by $(((x_1,p_1),...,(x_r,p_r)),(y_1,...,y_r))$ such that $p_1=\cdots=p_r$ and $y_1=\cdots=y_r$. To say that a map $M\times P\to N$ has the associated map 
$$
\Conf(r,M\times P)\longrightarrow \Conf(r,M\times P)\times N^r 
$$
which is transverse to $Z$ is the same as saying that the map
$$
\Conf(r,M)\times P\longrightarrow N^r
$$
is transverse to the thin diagonal. This does not mean that for every $p$ in $P$ the map $\Conf(r,M)\to N^r$ is transverse to the diagonal. But the case we are interested in is when transversality does not hit the diagonal at all. This happens if $\dim(\Conf(r,M)\times P)=rm+\dim(P)$ is less than the codimension of the diagonal in $N^r$, which is $nr-n=n(r-1)$. In other words, if $\dim(P)<(r-1)n-rm$, then arranging, by a small homotopy, for $\Conf(r,M)\times P\to N^r$ to be transverse to the diagonal means arranging for it to not hit the diagonal at all. Then for each point in $P$, the map $\Conf(r,M)\to N^r$ of course does not hit it either.

When $P=S^k$, we thus have that, if $k<(r-1)n-rm$, the map \eqref{E:AssociatedMap} is transverse to the diagonal, which in turn means that, for all $v\in S^k$, we know $h(v)=f$ is an $(r+1)$-immersion that does not map any $r$ points $x_1$, $x_2$, ..., $x_r$  in $M$ to the same point in $N$.  But this precisely means that $h$ factors through $\rImm(M,N)$, i.e.~the dotted arrow in \eqref{E:LiftTorImm} exists.

This argument also works relatively, i.e.~it can be repeated for maps 
$$
S^k\times I\longrightarrow (r+1)\!\Imm(M,N).
$$
The difference now is that the map induced by the adjoint is $\conf(r,M)\times S^k\times I\to N^r$ and so the codimension of the preimage of the diagonal is  $rm+k+1-(r-1)n$. This means that the map misses the diagonal if $k<(r-1)n-rm-1$.

What we have thus shown is that the homotopy classes of maps of $S^k$ to $(r+1)\!\Imm(M,N)$ and $\rImm(M,N)$ are in bijective correspondence if $k<(r-1)n-rm-1$; the bijection is induced by the inclusion  $\rImm(M,N)\to (r+1)\!\Imm(M,N)$ and, as we have just shown, lifts of maps $S^k\to(r+1)\!\Imm(M,N)$ and $S^k\times I\to\Imm(M,N)$ to $\rImm(M,N)$.  The inclusion thus induces isomorphisms
$$
\pi_k(\rImm(M,N))\stackrel{\cong}{\longrightarrow}\pi_k((r+1)\!\Imm(M,N)), \ \ \ \text{for } k<(r-1)n-rm-1.
$$
In addition, since we have a lift of maps $S^k\to\Imm(M,N)$ for $k<(r-1)n-rm$, and in particular for $k=(r-1)n-rm-1$, we thus have a surjection
$$
\pi_k(\rImm(M,N))\twoheadrightarrow\pi_k((r+1)\!\Imm(M,N)), \ \ \ \text{for } k=(r-1)n-rm-1.
$$
Putting this together means precisely that the inclusion $\rImm(M,N) \to (r+1)\!\Imm(M,N)$
is $((r-1)n-rm-1)$-connected. 
\end{proof}

\begin{cor}\label{C:rImm->ImmConnectivity}
With the assumptions as in \refT{rImm->(r+1)ImmConnectivity}, the inclusion
$$
\rImm(M,N) \longrightarrow \Imm(M,N)
$$
is $((r-1)n-rm-1)$-connected.
\end{cor}

\begin{proof}
The proof of \refT{rImm->(r+1)ImmConnectivity} goes through the same way with $(r+1)\!\Imm(M,N)$ replaced by $\Imm(M,N)$.  Alternatively, the connectivity number from \refT{rImm->(r+1)ImmConnectivity}, $((r-1)n-rm-1)$, does not decrease as $r$ goes to infinity since $m\leq n$ (a standing assumption since otherwise there are no immersions or embeddings of $M$ in $N$).  Since $\Imm(M,N)$ is the limit of the inclusions $\rImm(M,N) \to (r+1)\!\Imm(M,N)$ and since the least connectivity of those inclusions is hence $((r-1)n-rm-1)$, it follows by \refP{CompositionConnectivity} that $\rImm(M,N) \to\Imm(M,N)$ has the same connectivity.
\end{proof}

What follows immediately from \refC{rImm->ImmConnectivity} is that, as long as $(r-1)n-rm-1\geq 0$, the inclusion $\rImm(M,N) \to \Imm(M,N)$ is surjective on $\pi_0$.  In other words, we recover the following familiar result.
\begin{cor}
For $n>\frac{r}{r-1}m$, any immersion of $M$ in $N$ is regular homotopic (homotopic through immersions) to an immersion that has no $r$-fold self-intersections.
\end{cor}

\begin{rem}When $r=2$, namely when $\rImm(M,N)$ is the space of embeddings $\Emb(M,N)$, the previous result says that, for $n>2m$, any immersion of $M$ in $N$ is regular homotopic to an embedding of $M$ in $N$, and this is the well-known \emph{Whitney Easy Embedding Theorem}.
\end{rem}


\section{Configuration and $r$-configuration spaces}\label{S:Configurations}


The examples of embedding and $r$-immersion spaces that are most important for us are those of configuration spaces, which is the case when $M$ is a collection of points.  Let $\underline k=\{1,2,...,k\}$.  We then define the  \emph{configuration space of $k$ points in $N$} to be
\begin{equation}\label{E:ConfSpace}
\Conf(k,N):=\Emb(\underline k, N)\cong\{(x_1, x_2, ..., x_{k})\in N^{k} \colon x_i\neq x_j \text{ for } i\neq j\}.
\end{equation}
This space can be thought of as $N^k$ with all the diagonals (i.e.~the fat diagonal) removed.  A related space, and one that is central in this paper, is the  
\emph{$r$-configuration space of $k$ points in $N$} defined by
$$
\rConf(k,N):=\rImm(\underline k, N).
$$
This is the space of configurations of $k$ points in $N$ where at most $r-1$ of them can be equal.  In other words, this is $N^k$ with some, but not all, of the diagonals removed, and is for this reason sometimes called a \emph{partial configuration space}.  This space can also be thought of as the complement of the union of certain diagonals in $N^k$ and is hence an example of a complement of a \emph{subspace arrangement}.  

As illustrated by the following example, $r$-configuration spaces are simple in some cases, and it is precisely this simplicity that will allow us to describe the low stages of the Taylor tower for $\rImm(M,N)$ without too much difficulty.

\begin{example}\label{Ex:rImmersionsExamples}
If $k<r$, then 
\begin{equation}\label{E:k<rN}
\rConf(k, N)\cong N^k.
\end{equation}
This is because there is no restriction on the $k$ points being different, and so no diagonals are removed from $N^k$.  In particular, when $k<r$,
\begin{equation}\label{E:k<rR^n}
\rConf(k, \R^n)\cong (\R^n)^k\simeq \ast.
\end{equation}
If $k=r$, then 
\begin{equation}\label{E:k=rN}
\rConf(r, N)\cong N^r\setminus \Delta
\end{equation}
where $\Delta$ is as before the thin diagonal. This is because any proper subset of the $r$ points is allowed to equal in $\rImm(r, N)$, but not all of them. In particular, when $N=\R^n$, 
\begin{equation}\label{E:k=rR^n}
\rConf(r, \R^n)\cong (\R^n)^r\setminus \Delta \simeq S^{(r-1)n-1}.
\end{equation}
The homotopy equivalence is given by retracting $(\R^n)^r\setminus \Delta$ onto the orthogonal complement of $\Delta\setminus \{0\}$, which is $(\R^n)^{r-1}\setminus \{0\}$, and then normalizing to length one.
\end{example}

The space $\rConf(k,N)$ is in general  more difficult and less understood than $\Conf(k,N)$.  Its (co)homology is known (see Introduction of \cite{DT:Overlapping} for an overview of the literature dealing with the (co)homology of $\rConf(k,N)$), and it is known that its suspension is a wedge of spheres \cite[Corollary 3.10]{DT:Overlapping}.  In addition,  the connectivity and homotopy groups of $\rConf(k,N)$ through a range were studied in \cite{KS:DiagonalComplements}. From the point of view of subspace arrangements, the stable homotopy type of these spaces can be identified as 
the Spanier-Whitehead dual of
$$
\bigvee_{p\in P} \big(\Delta(P_{<p})\ast S^{d(p)-1}\big)
$$
where $P$ is the partition poset associated to the arrangement of the diagonals that have been removed and $d(p)$ is the dimension of the subspace corresponding to $p\in P$ \cite{ZZ:Arrangements}.

Spaces of $r$-configurations are central to the work here, as they are the building blocks for the Taylor tower for $\rImm(M,N)$.  Equally important are the projection maps between them and this is where much of the difficulty lies:  For ordinary configuration spaces, the projection maps
$$
\Conf(k+1,N) \longrightarrow \Conf(k,N)
$$
that forget a point are fibrations \cite{FN:ConfFibration}.  In the case $N=\R^n$, the fiber is even easily identified as $\R^n$ with $k-1$ points removed, which is homotopy equivalent to $\bigvee_{k}S^{n-1}$.

However, the corresponding projections of $r$-configuration spaces,
\begin{equation}\label{E:rConfProjection}
\rConf(k+1,N) \longrightarrow \rConf(k,N),
\end{equation}
are not fibrations, as illustrated in the following example.  

\begin{example}
Consider the projection
\begin{align*}
\operatorname{3Conf}(3,\R^n)& \longrightarrow \operatorname{3Conf}(2,\R^n) \\
(x_1,x_2,x_3) & \longmapsto (x_1,x_2)
\end{align*}
The fiber over a point $(x_1,x_2)\in\operatorname{3Conf}(2,\R^n)$ where $x_1\neq x_2$ is $\R^n$ since, in that fiber, $x_3$ can be anywhere in $\R^n$, including at $x_1$ or at $x_2$.  The fiber over a point $(x_1,x_2)$ where $x_1= x_2$ is $\R^n\setminus \{x_1\}\simeq S^{n-1}$ since $x_3$ can be anywhere except at the point $x_1= x_2$.  Since the fibers over two different points have different homotopy type, the map is not a fibration.
\end{example}

Understanding the connectivity of the projection \eqref{E:rConfProjection}, which is important for understanding the Taylor tower of $\rImm(M,N)$, therefore requires understanding its \emph{homotopy} fiber, and not just its fiber(s).  The former is unfortunately an unwieldy space.  In the case of $N=\R^n$, however, we at least have a handle on the connectivity of the projection map.

\begin{prop}\label{P:r+1R^n->rR^n}
Suppose $r\leq k<l$.  Then the map 
$$\rConf(l, \R^n)\longrightarrow \rConf(k, \R^n),
$$ 
given by forgetting $l-k$ configuration points, is $((r-1)n-1)$-connected.
\end{prop}


\begin{proof}
First look at the map 
\begin{equation}\label{E:HomologyPullbackMap}
\rConf(k, \R^n)\longrightarrow \rConf(r, \R^n)\simeq S^{(r-1)n-1}
\end{equation}
given by forgetting $k-r$ points (the equivalence comes from \eqref{E:k=rR^n}).
The space $\rConf(k, \R^n)$ is $((r-1)n-2)$-connected.   This follows from \refC{rImm->ImmConnectivity} with $M=\underline k$ and $N=\R^n$, in which case $\Imm(M,N)$ is contractible  (this result also appears as Corollary 4.7 in \cite{KS:DiagonalComplements}).  The sphere $S^{(r-1)n-1}$ has the same connectivity.  Now, the connectivity of the space $\rConf(k, \R^n)$, $k\geq r$, does not depend on $k$, so the same argument applies to the connectivity of the map \eqref{E:HomologyPullbackMap} with $l$ replacing $k$.

Furthermore, since $\rConf(k, \R^n)$ is a subspace arrangement, one can use the Goresky-MacPherson formulas \cite{GM:MorseTheory} to study thir cohomology, which is generated precisely by the spherical classes represented by the maps \eqref{E:HomologyPullbackMap}.  In particular, there is an injection
$$
\Ho^{(r-1)n-1}(\rConf(k, \R^n))\longrightarrow \Ho^{(r-1)n-1}(\rConf(l, \R^n))
$$
given by inclusion of generators.  Since everything is finitely generated, we then have a surjection on homology groups in the same degree, and, by the Hurewicz Theorem, a surjection on the first nontrivial homotopy group $\pi_{(r-1)n-1}$. 
%
\end{proof}

The above in particular provides the connectivity of the projection to one fewer points as in \eqref{E:rConfProjection}, but identifying the homotopy fiber of that map is more difficult.  The hope is that it is equivalent to a wedge of spheres, like the fiber for projections of ordinary configurations is. One can show that the \emph{suspension} of the homotopy fiber is indeed a wedge of spheres, which provides evidence that the homotopy fiber is as well.

In the case of $r$-configurations in arbitrary $N$, the situation is of course more complicated for various reasons, one of them being that we do not have as good of an understanding of the homology of this space.  The hope, however, is that the connectivity from \refP{r+1R^n->rR^n} remains the same (it does in the case of ordinary configuration spaces in $N$).


\section{Cubical diagrams and total fibers}\label{S:Cubes}


We will assume the reader is familiar with the language of homotopy limits, including homotopy fibers and homotopy pullbacks.  However, we will almost exclusively require these notions only in the case of cubical diagrams, and a source for that material is \cite{MV:Cubes} (foundational material on the subject can be found in \cite{CalcII}).

%
%
%
%

Let $\Top$ be the category of topological spaces and maps between them.  We will also sometimes use the same notation for the category of based spaces and maps and this will not cause confusion.

\begin{defin}\label{D:k-connectedSpace}\ 
\begin{itemize}
\item A nonempty space $X$ is \emph{$k$-connected} if $\pi_i(X,x)=0$ for all $0\leq i\leq k$ and for all choices of basepoint $x\in X$.  An infinitely connected space is \emph{weakly contractible}.
\item A map $f\colon X\to Y$ is \emph{$k$-connected} if its homotopy fiber (over any point $y\in Y$) is $(k-1)$-connected.  Equivalently, if $X\neq\emptyset$, $f$ is $k$-connected if, for all $x\in X$, the induced map
$$
f_*\colon \pi_i(X,x)\longrightarrow \pi_i(Y,f(x))
$$
is an isomorphism for all $i<k$ and a surjection for $i=k$.  An infinitely connected map is a \emph{weak equivalence}.
\end{itemize}
\end{defin}


 \begin{example}\ 
 \begin{itemize}
\item Path-connected spaces are 0-connected.
\item Simply-connected spaces are 1-connected.
\item The sphere $S^k$ is $(k-1)$-connected.  
\item A map between $k$-connected spaces is $(k-1)$-connected.
\item A map $X\to \ast$ is $k$-connected if and only if $X$ is $(k-1)$-connected.
\end{itemize}
 \end{example}
 
 For the proof of the following, see, for example, \cite[Proposition 2.6.15]{MV:Cubes}.

\begin{prop}\label{P:CompositionConnectivity}
Given maps $f\colon X\to Y$ and $g\colon Y\to Z$, 
\begin{itemize}
\item If $f$ and $g$ are $k$-connected, then $g\circ f$ is $k$-connected;
\item If $f$ is $(k-1)$-connected and $g\circ f$ is $k$-connected, then $g$ is $k$-connected;
\item If $g$ is $(k+1)$-connected and $g\circ f$ is $k$-connected, then $f$ is $k$-connected.
\end{itemize}
\end{prop}

Let as before $\underline k = \{1,2,...,k\}$ and denote by $\calP(\underline k)$ and $\calP_0(\underline k)$ the set of all subsets of $\underline k$ and the set of all nonempty subsets of $\underline k$, respectively.  Both of these can be regarded as a category (poset) with inclusions as morphisms.  

\begin{defin}\label{D:Cube}\ 
\begin{itemize}
\item A \emph{$k$-cube}, or a \emph{cubical diagram of dimension $k$} is a (covariant) functor
\begin{align*}
\calX\colon \calP(\underline k) & \longrightarrow \Top \\
S &\longmapsto X_S
\end{align*}
\item A \emph{punctured $k$-cube}, or a \emph{punctured cubical diagram of dimension $k$} is the same except the domain is $\calP_0(\underline k)$.
\end{itemize}
\end{defin}

A 0-cube is a space, a 1-cube is a map of spaces, a 2-cube is a commutative square of spaces, etc.  A $k$-cube, for $k\geq 1$, can be regarded as a map, i.e.~a natural transformation, of $(k-1)$-cubes.  So a 1-cube is a map of 0-cubes (spaces),  a 2-cube (square) is a map of 1-cubes (maps), a 3-cube is a map of 2-cube (squares), and so on.

Removing the initial space $X_\emptyset$ from a cube of spaces leaves a punctured cube of spaces.  Furthermore, because $X_\emptyset$ maps into the rest of the cube, one also has a canonical induced map (see \cite[Section 5.4]{MV:Cubes} for details)
$$
a(\calX)\colon X_\emptyset \longrightarrow \underset{S\in \calP_0(\underline k)}{\holim} X_S
$$
where $\holim$ denotes the homotopy limit.

\begin{defin}\label{D:CartesianCube}
A $k$-cube is said to be $c$-cartesian if $a(\calX)$ is $c$-connected.  If $c=\infty$, the cube is \emph{(homotopy) cartesian}.
\end{defin}

A cube is \emph{based} if all the spaces in it are based and basepoints map to basepoints.  One way to base a cube is to choose a basepoint in $X_\emptyset$ and let it determine basepoints in the rest of the cube.  If a cube is based, so is the punctured cube associated to it, and this produces a natuaral basepoint in $ \underset{S\in\calP_0(\underline k)}{\holim} X_S$.

\begin{defin}\label{D:TotalFiber}
The \emph{total (homotopy) fiber of a based cube $\calX$}, $\tfiber(\calX)$, is 
$$
\tfiber(\calX)=\hofiber(a(\calX)),
$$
where the homotopy fiber is taken over the natural basepoint.
\end{defin}

A cube is thus $c$-cartesian if its total fiber is $(c-1)$-connected.

There is another convenient description of the total fiber of a cube in terms of \emph{iterated fibers}.  For the proof of the following, see \cite[Proposition 5.5.4]{MV:Cubes}.

\begin{prop}\label{P:IteratedFiber}
Let $\calX$ be a based cube.  For $\underline k=\emptyset$, we have $\tfiber(\calX)=X_\emptyset$.  For $\underline k\neq\emptyset$, regard $\calX$ as a map of $(k-1)$-cubes, $\calY\to\calZ$.  Then
$$
\tfiber(\calX)=\hofiber(\tfiber(\calY)\longrightarrow \tfiber(\calZ)).
$$
\end{prop}

\begin{example}\label{Ex:ProductCubes1}
Let $X_1, ..., X_k$ be spaces.  Consider the cube $\calX\colon \calP(\underline k)\to\Top$ given by $\calX(S)=\prod_{i\notin S}X_i$ for $S\neq \underline k$ and $\calX(\underline k)=\ast$.  Then this cube is homotopy cartesian.  One way to see this is that each square face is homotopy cartesian, i.e.~a homotopy pullback, which implies that the cube is as well.  For details, see \cite[Example 5.4.21]{MV:Cubes}.
\end{example}

\begin{example}\label{Ex:ProductCubes2}
A related example is the $k$-cube where $\calX(\emptyset)=\prod_{i\in \underline k} X_i$, $\calX(\{i\})= X_i$, and $\calX(S)=\ast$ for all other $S\subset \underline k$.  This cube is also homotopy cartesian because regarding it as a map of cubes and  taking the fiber produces a $(k-1)$-cube of the sort from the previous example.  Since that cube is homotopy cartesian, so is this one.
\end{example}

For more on total homotopy fibers, see \cite[Section 5.4]{MV:Cubes}.


\section{Manifold calculus of functors and the Taylor tower for embeddings}\label{S:ManifoldCalculus}





In this section we review some of the main features of manifold calculus of functors.  For details, see \cite[Sections 10.2 and 10.3]{MV:Cubes} and \cite{M:MfldCalc}, in addition to the foundational papers \cite{GW:EI2, W:EI1}.  We will throughout pay attention to the functor $\Emb(M,N)$ since this is the one we wish to emulate in our analysis of $\rImm(M,N)$ in \refS{r-ImmersionsTaylorTower}.

Let $M$ and $N$ be smooth manifolds of dimensions $m$ and $n$, respectively.  Let
\begin{align*}
\mathcal{O}(M)= & \text{ category (poset) of open subsets of  $M$  with inclusions as morphisms}.
\end{align*}

Manifold calculus studies contravariant functors
$$
F\colon \mathcal{O}(M)\longrightarrow \Top
$$
that are 
\begin{itemize}
\item \emph{finitary}, namely for a sequence of open subsets $U_0\subset U_1\subset\cdots$, the canonical map
$
F(\cup_i U_i)\to \holim_i F(U_i)
$ 
is a homotopy equivalence; and
\item \emph{isotopy functors}, namely they take isotopy equivalences to homotopy equivalences
\end{itemize}

\begin{example}
Functors $\Imm(-,N)$, $\rImm(-,N)$, and $\Emb(-,N)$ are all finitary isotopy functors (see \cite[Proposition 1.4]{W:EI1} for $\Imm(-,N)$ and $\Emb(-,N)$; the argument for $\rImm(-,N)$ is same as for $\Imm(-,N)$).  They are contravariant on $\calO(M)$ since an inclusion gives a restriction map ``going the other way".
\end{example}


Now let $\mathcal{O}_k(-)$ be the subcategory $\mathcal{O}(-)$ consisting of open subsets of $M$ diffeomorphic to up to $k$ disjoint balls.  

\begin{defin}\label{D:TaylorStage}
Let $F$ be a finitary isotopy functor.  For $U\in \mathcal{O}(M)$, the $k$th stage of the Taylor tower is defined as
\begin{equation}\label{E:TaylorStage}
T_kF(U)=\underset{V\in \mathcal{O}_k(U)}{\holim} F(V).
\end{equation}
\end{defin}
The above homotopy limit is in some sense trying to reconstruct $F(U)$ from information about collections of its open balls (in category theory language, this is a \emph{homotopy right Kan extension}). 

We then get the \emph{Taylor tower of $F$} consisting of functors $T_kF$ with natural transformations between them and admitting a natural transformation from $F$:
$$
F(-)\longrightarrow\big(T_\infty F(-)\to\cdots\to T_kF(-)\to T_{k-1}F(-)\to\cdots \to T_0F(-)  \big).
$$
The transformations between the stages are induced by inclusions $\calO_{k-1}(U)\to\calO_{k}(U)$ and the transformation from $F$ to the \emph{stages} $T_kF$ of the tower by inclusions $\calO_{k}(U)\to\calO(U)$.  The functor $T_{\infty}F(-)$ is the inverse limit of the tower. 


Evaluating this diagram on $U\in\calO(M)$, we get a diagram of spaces with maps between the stages that are fibrations.

The definition of $T_kF(-)$ is not easy to work with.  But there is an alternative way to think about $T_kF$ in terms of cubical diagrams (at the expense of losing some functoriality properties).  We will not need this construction here, but the details of how to go from \refD{TaylorStage} to this cubical model can be found in \cite[Example 10.2.18]{MV:Cubes}.

The stages $T_kF(-)$ are polynomial in the following sense.

\begin{defin}\label{D:Derivative}
Let $B_1, ..., B_k$ be pairwise disjoint open balls in $M$.  Let $\calX$ be the $k$-cube given by
\begin{align*}
\calX \colon \calP(\underline k) & \longrightarrow \Top\\
S   & \longmapsto F\Big(\bigcup_{i\notin S}B_i\Big)
\end{align*}
Then define the \emph{$k^{th}$ derivative of $F$ at the empty set}, denoted by $F^{(k)}(\emptyset)$, to be the total fiber of $\calX$.
\end{defin}

\begin{defin}\label{D:Polynomial}
A contravariant functor $F\colon \calO(M)\to\Top$ is \emph{polynomial of degree $\leq k$} if, for all $U\in \calO(M)$ and for all pairwise disjoint nonempty closed subsets $A_1, ..., A_{k+1}$ of $U$, the cube 
\begin{align*}
\calX \colon \calP(\underline{k+1}) & \longrightarrow \Top\\
S   & \longmapsto F\Big(U-\bigcup_{i \in S}A_i\Big)
\end{align*}
is homotopy cartesian.
\end{defin}

The two definitions above are related by setting $U$ to be $k+1$ disjoint open balls and $A_i$ to be its components.  What falls out of the definitions then is that a polynomial functor of degree $\leq k$ has contractible derivatives of higher order.

\begin{example}\label{Ex:Polynomials}
The immersion functor $\Imm(-,N)$ is polynomial of degree $\leq 1$, or \emph{linear} \cite[Example 2.3]{W:EI1}.  The embedding functor $\Emb(-,N)$ is not polynomial of degree $\leq k$ for any $k$ \cite[Example 4.7]{M:MfldCalc}.
\end{example}

\begin{thm}[\cite{W:EI1}, Theorem 3.9]\label{T:T_kPolynomial}
Let $F\colon \calO(M)\to\Top$ be a contravariant finitary isotopy functor. The $k^{th}$ stage  $T_kF$ of the Taylor tower for $F$ is polynomial of degree $\leq k$.
\end{thm}

Polynomial functors can be characterized by what they do on balls.

\begin{thm}[\cite{W:EI1}, Theorem 5.1]\label{T:PolynomialCharacterization}
Suppose $F$ and $G$ are contravariant finitary isotopy functors from $\calO(M)$ to $\Top$ that are polynomial of degree $\leq k$.  Suppose $T\colon F\to G$ is a natural transformation that is an equivalence for all $U\in \calO_k(M)$.  Then $T$ is an equivalence for all $U\in \calO(M)$.
\end{thm}


The proof of the next result combines \refT{PolynomialCharacterization} with the fact that embeddings and immersions agree on a single ball up to homotopy.  We will repeat this argument when we deduce that the linearization of the space of $r$-immersions is also the space of immersions.

\begin{prop}[\cite{W:EI1}, bottom of page 97]\label{P:T_1Emb=Imm}
The linearization of the space of embeddings is the space of immersions, namely there is an equivalence
$$
T_1\Emb(-,N)\simeq \Imm(-,N).
$$

\end{prop}

There are two natural questions one can ask about the Taylor tower:
\begin{enumerate}
\item Does the Taylor tower for $F$ converge?
\item Does it converge to $F$?
\end{enumerate}

The convergence of the tower means that the connectivity of the maps $T_kF\to T_{k-1}F$ grows to infinity with $k$.  This can be established by looking at the spaces
\begin{equation}\label{E:Layer}
L_kF(-):=\hofiber (T_kF(-) \longrightarrow T_{k-1} F(-))
\end{equation}
and showing that their connectivity grows with $k$. Here we need to be working with a based Taylor tower, which can be accomplished by choosing a basepoint in $F(M)$ which in turn bases $T_k(U)$ for all $k$ and $U$.  The functor $L_kF(-)$ is called the \emph{$k^{th}$ layer} of the Taylor tower of $F$ and is \emph{homogeneous} in the sense that all its derivatives of degree $<k$ are trivial.  One of the main results in manifold calculus of functors is the statement that classifies all homogeneous functors in terms of certain spaces of sections \cite[Theorem 8.5]{W:EI1}.  We will not need this result here, but will use the following consequence (see \cite[Proposition 5.1.1]{M:MfldCalc} for more details).

\begin{prop}\label{P:LayerAndDerivative}
Let $M$ be of dimension $m$ and suppose $F\colon \calO(M)\to\Top$ is a contravariant finitary isotopy functor.  If the derivative $F^{(k)}(\emptyset)$ (see \refD{Derivative}) is $c_k$-connected, then $L_kF(M)$ is $(c_k-km)$-connected and, consequently, the map $T_k(M)\to T_{k-1}(M)$ is $(c_k-km+1)$-connected.
\end{prop}

The above is true when $M$ is compact and is replaced by any interior of a compact codimension zero handlebody,
in which case the handle dimension (the highest dimension of a handle
necessary to build the handlebody) replaces $m$.
 However, we only focus on the case of $M$ itself, and this will remain true for the rest of the paper and for our results.  This is fine since, in most applications of manifold calculus, this is the only case of importance.

For the embedding functor in particular, we have the following result.  

\begin{thm}\label{T:EmbeddingsLayer}
The derivative $\Emb(M,N)^{(k)}(\emptyset)$ is $(k-1)(n-2)$-connected, and hence the connectivity of the map $T_k\Emb(M,N)\to T_{k-1}\Emb(M,N)$ is 
$$
(k-1)(n-2)-km+1=(k-1)(n-m-2)-m+1.
$$
\end{thm}


For the proof of \refT{EmbeddingsLayer}, see, for example, \cite[Theorem 10.3.3]{MV:Cubes}.   In brief, one first passes from embeddings of balls, which are in the definition of the derivative, to embeddings of points, namely configuration spaces, and this is fine since balls are homotopy equivalent to points.  The proof then comes down to showing that the cube of configuration spaces and projections between them is $((k-1)(n-2)+1)$-cartesian, and hence the total fiber, i.e.~the derivative, has connectivity one less (one reference for this is \cite[Example 6.2.9]{MV:Cubes}).    The key is that projection maps between configuration spaces are fibrations, so that their (homotopy) fiber is easy to handle.  More will be said about this proof, and the ways in which the case of $r$-immersions is more complicated, in the discussion at the end of the paper.

The second question about the convergence of the Taylor tower, namely whether the tower converges to $F$, is a more difficult one.  By convergence to $F$ we mean that the map $F\to T_\infty F=\holim_k T_k$ is an equivalence, i.e.~infinitely connected.  One way to establish this would be to argue that the connectivity of the maps $F\to T_kF$ grows with $k$.  This is precisely what happens with the embedding functor.  Namely, we have the following result from \cite{GW:EI2}, which builds heavily on the work in \cite{GK}.

\begin{thm}\label{T:EmbeddingsConvergence}
Suppose $M$ is a smooth closed manifold of dimension $m$, $N$ is a smooth manifold of dimension $n$, and $n-m\geq 2$. Then the map
$$
\Emb(M,N)\longrightarrow T_k\Emb(M,N)
$$
is $(k(n-m-2)-m+1)$-connected.  If $n-m>2$, then the connectivities grow with $k$ and the Taylor tower therefore converges to $\Emb(M,N)$.
\end{thm}

In line with the comments following \refP{LayerAndDerivative}, the above is true when $M$ is replaced by the interior of a codimension zero handlebody, in which case $m$ has to be replaced by its handle dimension.  The proof is difficult and requires various disjunction results for embeddings; for an overview, see \cite[Section 10.3.2]{MV:Cubes}.

\begin{rem}\label{R:ExpectSameNumber}
The number in \refT{EmbeddingsConvergence} is the same as the number in \refT{EmbeddingsLayer} (with a shift in $k$).  This is not surprising since the first number can be used to conjecture the second:  Suppose we know that the Taylor tower converges, i.e.~the connectivities of the maps between the stages increase, and that it converges to $F$.  Suppose $T_{k+1}F(M)\to T_k F(M)$ is $c$-connected.  Then we have a diagram
$$
\xymatrix{
F(M)\ar[r]^-\sim \ar[dr] & T_\infty F(M) \ar[d]^-{\text{$c$-connected}}\\
& T_k F(M)
}
$$
The vertical map in the above diagram is also $c$-connected, due to \refP{CompositionConnectivity}, because it is the composition of maps for which the least connectivity is $c$.  Again by \refP{CompositionConnectivity}, it then follows that the connectivity of $F(M)\to T_k F(M)$ must also be $c$.  For $F(M)=\Emb(M,N)$, $c$ is precisely $k(n-m-2)-m+1$.
\end{rem}


\section{First $r$ stages of the Taylor tower for $r$-immersions}\label{S:r-ImmersionsTaylorTower}


In this section we give the description of the connectivities between the first $r$ stages of the Taylor tower for $r$-immersions.  Much of what we do can be adapted from the case of $M$ to the case of a codimension zero handlebody in $M$ (see comments following \refP{LayerAndDerivative} and \refT{EmbeddingsConvergence}), but we content ourselves with the case of $M$ because that is the most important and useful one.  Additionally, the generalization would require us to venture outside the intended scope and level of difficulty of this paper.  The general case will be tackled in future work.

 We start with $T_1\rImm(M,N)$.  The following is analogous to \refP{T_1Emb=Imm}.

\begin{prop}\label{P:T_1rImm=Imm}
The linearization of the space of $r$-immersions is the space of immersions, namely there is an equivalence
$$
T_1\rImm(M,N)\simeq \Imm(M,N).
$$
\end{prop}

\begin{proof}
We already know from \refEx{Polynomials} that $\Imm(M,N)$ is linear and from \refT{T_kPolynomial} that $T_1\rImm(M,N)$ is linear.  Therefore by \refT{PolynomialCharacterization}, to prove the desired equivalence it suffices to check that that the equivalence holds for a single ball.  Namely, with $B^m$ the $m$-dimensional ball, we want to show that
$$
\rImm(B^m,N)\simeq \Imm(B^m,N).
$$
But the argument for this is the same as showing that spaces of embeddings and immersions agree on a single ball.  Namely, one can differentiate an immersion or an $r$-immersion at, say, the center of the ball, to get a point in the Stiefel manifold of $m$-frames in the tangent space of $N$.  The fibers of both maps are contractible which means that $\rImm(B^m,N)$ and $\Imm(B^m,N)$ are equivalent.
\end{proof}

We next look at the layers of the Taylor tower for $\rImm(M,N)$, which, by \refP{LayerAndDerivative}, will require us to understand the derivatives $\rImm(M,N)^{(k)}(\emptyset)$.   Looking back at \refD{Derivative}, the derivatives are total fibers of cubical diagrams consisting of spaces of $r$-immersions of unions of balls.  But, as alluded to in the discussion following \refT{EmbeddingsLayer}, it is possible to replace these with diagrams of $r$-immersions of points, namely with diagrams of $\rConf(k,N)$.  The argument for this is identical to that for the case of embeddings.  For details, see discussion following the statement of \cite[Theorem 7.2]{M:MfldCalc}.

Suppose the Taylor tower for $\rImm(M,N)$ has been based (see comment after \eqref{E:Layer}).  Recalling from \refD{k-connectedSpace} that a weakly contractible space is one that is infinitely connected, we have the following result.

\begin{prop}\label{P:L_krImmContractible}
For $2\leq k\leq r-1$, the layer $L_k\rImm(M,N)$ of the Taylor tower for $\rImm(M,N)$ is weakly contractible.
\end{prop}

\begin{proof}
It suffices to show that the derivative $\rImm^{(k)}(\emptyset,N)$ is weakly contractible.  Then, by \refP{LayerAndDerivative}, it follows that $L_k\rImm(M,N)$ is weakly contractible as well.

By the discussion above, $\rImm^{(k)}(\emptyset,N)$ is the total fiber of the cubical diagram 
\begin{align*}
\calX \colon \calP(\underline k) & \longrightarrow \Top\\
S   & \longmapsto \rConf(k-|S|, N)
\end{align*}
where the maps are projections given by forgetting points in the configuration, namely an inclusion $T\to S$ gives a map that projects away from those configuration points indexed by $S-T$.  However, since $k<r$, as observed in \refEx{rImmersionsExamples}, 
$$
\rConf(k-|S|, N)=N^{k-|S|}
$$
The cube in question is thus equivalent to the cube 
\begin{align}\label{E:ProductsCube}
\calX \colon \calP(\underline k) & \longrightarrow \Top\\
S   & \longmapsto \begin{cases} N^{k-|S|}, & S\neq \underline k; \\  
\ast, &  S=\underline k
\end{cases}
\end{align}
with projection maps as before projecting away from those factors indexed by $S-T$.  But this cube is homotopy cartesian by \refEx{ProductCubes1}.  This means precisely that its total fiber is weakly contractible.  Equivalent cubes have equivalent total fibers, and so the total fiber of the original cube, namely $\rImm^{(k)}(\emptyset,N)$, is also weakly contractible.
%
%
\end{proof}



\begin{thm}\label{T:T_kSameFork<r} For $2\leq k\leq r-1$, we have a weak equivalence
$$
T_k\rImm(M,N)\stackrel{\sim}{\longrightarrow}T_{k-1}\rImm(M,N).
$$
\end{thm}

\begin{proof}
By \refP{L_krImmContractible}, the fiber $L_k\rImm(M,N)$ of the fibration $T_k\rImm(M,N)\to T_{k-1}\rImm(M,N)$ is weakly contractible in the given range. By the homotopy long exact sequence of a fibration, it follows that the map $T_k\rImm(M,N)\to T_{k-1}\rImm(M,N)$ is infinitely connected, or a weak equivalence.
\end{proof}


\begin{cor}\label{C:Connectivityk<r}
For $1\leq k\leq r-1$, the map
$$
\rImm(M,N) \longrightarrow T_k\rImm(M,N)
$$
is $((r-1)n-rm-1)$-connected.
\end{cor}

\begin{proof}
We know from \refC{rImm->ImmConnectivity} and \refP{T_1rImm=Imm} that the map 
$$
\rImm(M,N) \longrightarrow T_1\rImm(M,N)\simeq \Imm(M,N)
$$
is $((r-1)n-rm-1)$-connected.  Inducting up the tower and using \refP{CompositionConnectivity} along with \refT{T_kSameFork<r} gives the desired result.
\end{proof}

Proposition~\ref{P:L_rrImmConnectivity} will give the connectivity of the next layer, $L_r\rImm(M,N)$.  Before we prove it, we need the following result.

\begin{lemma}\label{L:N-diagonal->N}
For $N$ a manifold of dimension $n$, the inclusion 
$$
N^r\setminus \Delta\longrightarrow N^r,
$$
where $\Delta$ is the thin diagonal, is $((r-1)n-1)$-connected.
\end{lemma}

\begin{proof}
This proof is a simpler version of the one from \refT{rImm->(r+1)ImmConnectivity}.  Namely, a map $S^k\to N^r$ by transversality generically misses the thin diagonal if $k<(r-1)n$, and a map $S^k\times I\to N^r$ misses it if $k<(r-1)n-1$.  This means that, under those dimensional assumptions, the maps are homotopic to maps that lift to $N^r\setminus \Delta$, and this means that the map $N^r\setminus\Delta\to N^r$ induces isomorphisms on homotopy groups $\pi_k$ for $k<(r-1)n-1$ and a surjection on $\pi_{(r-1)n-1}$, which is what we wanted to show.
%
\end{proof}

\begin{prop}\label{P:L_rrImmConnectivity}
The layer $L_r\rImm(M,N)$ of the Taylor tower for $\rImm(M,N)$ is $((r-1)n-rm-2)$-connected.
\end{prop}

\begin{rem}\label{R:Convergence=>T_r->T_r-1}
If we knew that the Taylor tower converged, \refP{L_rrImmConnectivity} would in fact be immediate from what we know already.  As we know from \refR{ExpectSameNumber}, convergence would mean that the connectivity of $\rImm(M,N)\to T_r\rImm(M,N)$ is greater than that of $\rImm(M,N)\to T_{r-1}\rImm(M,N)$, which is $(r-1)n-rm-1$ from \refC{Connectivityk<r}.  We would thus have a diagram
$$
\xymatrix{
\rImm(M,N)\ar[rr]^-{> (r-1)n-rm-1} \ar[drr]^-{(r-1)n-rm-1} & &T_r\rImm(M,N)\ar[d] \\
& &T_{r-1}\rImm(M,N)
}
$$
and it would then follow from \refP{CompositionConnectivity} that $T_r\rImm(M,N)\to T_{r-1}\rImm(M,N)$ is $((r-1)n-rm-1)$-connected.  But this is the same as saying that the fiber of this map, namely $L_r\rImm(M,N)$, is $((r-1)n-rm-2)$-connected.
\end{rem}

\begin{proof}[Proof of \refP{L_rrImmConnectivity}]
We start as in the proof of \refP{L_krImmContractible}, and so  $\rImm^{(r)}(\emptyset,N)$ is the total fiber of the cubical diagram 
\begin{align*}
\calX \colon \calP(\underline r) & \longrightarrow \Top\\
S   & \longmapsto \rConf(r-|S|, N)
\end{align*}
with projections as before.  The initial space in the cube, by \refEx{rImmersionsExamples}, is
$$
\rConf(r-|\emptyset|, N)=\rConf(r, N)=N^r\setminus \Delta
$$
The other spaces are, as in the proof of \refP{L_krImmContractible}, $\rConf(r-|S|, N)=N^{r-|S|}$ since then $r-|S|<r$.  So for example, when $r=3$, the cube $\calX$ is equivalent to 
\begin{equation}
\xymatrix
@=5pt
{
N^3\setminus \Delta \ar[rr]\ar[dd]\ar[dr] & & N^2\ar[dr]\ar'[d][dd] &\\
 & N^2\ar[rr]\ar[dd] & & N\ar[dd]\\
N^2\ar'[r][rr]\ar[rd] & & N\ar[rd] &\\
 & N\ar[rr] & & \ast
}
\end{equation}

Thus  
$$
\rImm^{(r)}(\emptyset,N)=\tfiber\calX\simeq\hofiber\left(N^r\setminus \Delta\to \underset{S\in \calP_0(\underline r)}{\holim}N^{r-|S|}\right)
$$

However, the homotopy limit of the punctured cube on the right is $N^r$.  This follows from \refEx{ProductCubes1} by setting all $X_i=N$ (and $k=r$); that example says precisely that there is an equivalence
$$
N^r\stackrel{\sim}{\longrightarrow} \underset{S\in \calP_0(\underline r)}{\holim}N^{r-|S|}.
$$
So now we have reduced the problem to finding the connectivity of the inclusion
$$
N^r\setminus \Delta\longrightarrow N^r.
$$
But by \refL{N-diagonal->N}, this map is $((r-1)n-1)$-connected, which means that its homotopy fiber, namely 
$\rImm^{(r)}(\emptyset,N)$ is $((r-1)n-2)$-connected.  It then follows by \refP{LayerAndDerivative} that the connectivity of $L_r\rImm(M,N)$ is
$$
(r-1)n-2 - rm = (r-1)n- rm-2.
$$
\end{proof}

\begin{rem}
When $N=\R^n$, which is often of most interest, then the above proof simplifies as follows:
The initial space in the cube is 
$$
\rConf(r, \R^n)=(\R^n)^r\setminus \Delta\simeq S^{(r-1)n-1},
$$
while the other spaces are products of $\R^n$ with itself and hence contractible.  When $r=3$ for example, we want the total fiber of the cube
\begin{equation}
\xymatrix
@=10pt
{
S^{2n-1} \ar[rr]\ar[dd]\ar[dr] & & \ast\ar[dr]\ar'[d][dd] &\\
 & \ast\ar[rr]\ar[dd] & & \ast\ar[dd]\\
\ast\ar'[r][rr]\ar[rd] & & \ast\ar[rd] &\\
 & \ast\ar[rr] & & \ast
}
\end{equation}
But this total fiber is simply $S^{2n-1}$, which is $(2n-2)$-connected, or in the general case of $S^{(r-1)n-1}$, $((r-1)n-2)$-connected, as desired.
\end{rem}


\section{Some comments about the higher stages}\label{S:HigherStages}


We end by saying a few words about the strategy for finding the connectivities of the higher  layers $L_{k+1}\rImm(M,N)$, when $k+1>r$ (we are indexing by $k+1$ rather than $k$ since the numbers will ultimately come out easier that way).  These layers exhibit genuinely different and more difficult behavior.  We will focus on the case $N=\R^n$, as more is known about partial configuration spaces in this situation.  Details will appear in the thesis work of the second author.

As before, to get at the connectivity of $L_{k+1}\rImm(M,\R^n)$, 
one would first establish the connectivity of the derivative $\rImm^{(k+1)}(\emptyset,\R^n)$.  This is the total fiber of the cube
\begin{align*}
\calX \colon \calP(\underline{k+1}) & \longrightarrow \Top\\
S   & \longmapsto X_S=\rConf((k+1)-|S|, \R^n)
\end{align*}
with projection maps.
The total fiber is
\begin{equation}\label{E:r+1TotalFiber}
\rImm^{(k+1)}(\emptyset,\R^n)=\tfiber\calX\simeq \hofiber\left(X_\emptyset \to \underset{S\in \calP_0(\underline{k+1})}{\holim} X_S \right)
\end{equation}

\begin{rem}
In the case $k+1=r+1$, the homotopy limit in above is simply $(S^{(r-1)n-1})^{r+1}$ (using \refEx{ProductCubes2}).
The map from \eqref{E:r+1TotalFiber} whose homotopy fiber we wish to understand is therefore equivalent to the map
\begin{equation}\label{E:r+1ProductMap}
\rConf(r+1, \R^n) \longrightarrow (S^{(r-1)n-1})^{r+1}.
\end{equation}
One way to get the connectivity of this map is to look at the homology of the two spaces using \cite{GM:MorseTheory}. 
\end{rem}

The idea now is to  emulate the proof in the case of  $\Emb(M,\R^n)$ and use the Blakers-Massey Theorem for cubes \cite[Theorem 2.5]{CalcII} (see also \cite[Section 6.2]{MV:Cubes}).
There are two pieces that would be required: 
\begin{enumerate}
\item The connectivites of the projection maps between $r$-configuration spaces, and
\item  One would take the homotopy fibers in one direction in $\calX$ and then look at each square face of the resulting $k$-cube, i.e.~look at the top square in the diagram
{\small
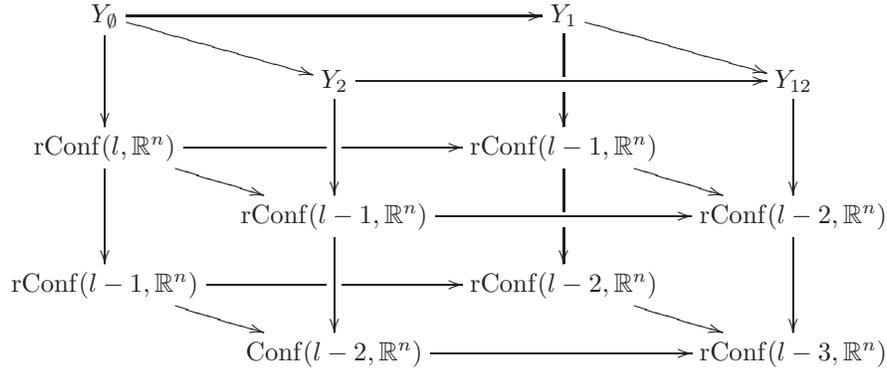
$$
\xymatrix@=10pt{
 Y_\emptyset \ar[dd]\ar[dr]\ar[rr]    &           &   Y_1 \ar'[d][dd]           \ar[dr]  &                  \\
        &  Y_2 \ar[rr] \ar[dd]  &             & Y_{12}   \ar[dd] \\
  \rConf(l,\R^{n}) \ar[dd]\ar[dr]\ar'[r][rr]    &           &   \rConf(l-1,\R^{n}) \ar'[d][dd]           \ar[dr]  &                  \\
         &  \rConf(l-1,\R^{n}) \ar[rr] \ar[dd]  &             & \rConf(l-2,\R^{n})   \ar[dd] \\
 \rConf(l-1,\R^{n}) \ar'[r][rr] \ar[dr] &        &   \rConf(l-2,\R^{n})   \ar[dr] &                   \\
      &   \Conf(l-2,\R^{n}) \ar[rr]      &                    &   \rConf(l-3,\R^{n})
}
$$ 
}
for various $3\leq l\leq k+1$.  One that proves that each such square is \emph{homotopy cocartesian}, i.e.~the map from the homotopy colimit of the diagram obtained by removing the final space in the square to the final space is an equivalence.  
\end{enumerate}
The first item is taken care of by \refP{r+1R^n->rR^n}, but the second is more difficult.  In the case we are trying to mimic, that of $\Emb(M,\R^n)$, one is aided by the fact that the corresponding cube consists of ordinary configuration spaces and that the projection maps between them are fibrations.  The homotopy fibers are just fibers, and are easy to understand (they are wedges of spheres).  The squares of fibers turn out to be simple, built out of an intersection and a union (see \cite[Example 3.7.5]{MV:Cubes}).
%
%
%
But as we already discussed in \refS{Configurations}, the projections of $r$-configurations are not fibrations, so one has to examine their homotopy fibers directly.

For general $N$, the situation is even more difficult since we do not know much about the connectivity or the homology of the spaces $\rConf(k,N)$.

Furthermore, without knowing that the tower converges, the connectivity of $\rImm(M,\R^n)\to T_{k}\rImm(M,\R^n)$ is harder to obtain.  For embeddings, deep \emph{disjunction} results \cite{GK} (see also \cite[Section 10.3.2]{MV:Cubes} for an overview) are required, and generalizing these to $r$-immersions is likely difficult. 

\bibliographystyle{alpha}

\bibliography{/Users/ismar/Dropbox/Bibliography}

\def\cprime{$'$} \def\cprime{$'$}
\begin{thebibliography}{BdBW18}

\bibitem[ALTV08]{ALTV}
Greg Arone, Pascal Lambrechts, Victor Turchin, and Ismar Voli{\'c}.
\newblock Coformality and rational homotopy groups of spaces of long knots.
\newblock {\em Math. Res. Lett.}, 15(1):1--14, 2008.

\bibitem[ALV07]{ALV}
Gregory Arone, Pascal Lambrechts, and Ismar Voli{\'c}.
\newblock Calculus of functors, operad formality, and rational homology of
  embedding spaces.
\newblock {\em Acta Math.}, 199(2):153--198, 2007.

\bibitem[AT14]{AT:LongPlanes2}
Gregory Arone and Victor Turchin.
\newblock On the rational homology of high-dimensional analogues of spaces of
  long knots.
\newblock {\em Geom. Topol.}, 18(3):1261--1322, 2014.

\bibitem[BdBW13]{BW:CalculusSheaves}
Pedro Boavida~de Brito and Michael Weiss.
\newblock Manifold calculus and homotopy sheaves.
\newblock {\em Homology Homotopy Appl.}, 15(2):361--383, 2013.

\bibitem[BdBW18]{BW:EmbConfCat}
Pedro Boavida~de Brito and Michael Weiss.
\newblock Spaces of smooth embeddings and configuration categories.
\newblock {\em J. Topol.}, 11(1):65--143, 2018.

\bibitem[DH12]{DwyerHess:LongKnots}
William Dwyer and Kathryn Hess.
\newblock Long knots and maps between operads.
\newblock {\em Geom. Topol.}, 16(2):919--955, 2012.

\bibitem[DT15]{DT:Overlapping}
Natalya Dobrinskaya and Victor Turchin.
\newblock Homology of non-{$k$}-overlapping discs.
\newblock {\em Homology Homotopy Appl.}, 17(2):261--290, 2015.

\bibitem[FN62]{FN:ConfFibration}
Edward Fadell and Lee Neuwirth.
\newblock Configuration spaces.
\newblock {\em Math. Scand.}, 10:111--118, 1962.

\bibitem[Fri15]{Frick:Tverberg}
Florian Frick.
\newblock Counterexamples to the topological {T}verberg conjecture.
\newblock arXiv:1502.00947, 2015.

\bibitem[GK15]{GK}
Thomas~G. Goodwillie and John~R. Klein.
\newblock Multiple disjunction for spaces of smooth embeddings.
\newblock {\em J. Topol.}, 8(3):651--674, 2015.

\bibitem[GKW01]{GKW:EmbDisjSurg}
Thomas~G. Goodwillie, John~R. Klein, and Michael~S. Weiss.
\newblock Spaces of smooth embeddings, disjunction and surgery.
\newblock In {\em Surveys on surgery theory, {V}ol. 2}, volume 149 of {\em Ann.
  of Math. Stud.}, pages 221--284. Princeton Univ. Press, Princeton, NJ, 2001.

\bibitem[GM88]{GM:MorseTheory}
Mark Goresky and Robert MacPherson.
\newblock {\em Stratified {M}orse theory}, volume~14 of {\em Ergebnisse der
  Mathematik und ihrer Grenzgebiete (3) [Results in Mathematics and Related
  Areas (3)]}.
\newblock Springer-Verlag, Berlin, 1988.

\bibitem[GM10]{GM:LinksEstimates}
Thomas~G. Goodwillie and Brian~A. Munson.
\newblock A stable range description of the space of link maps.
\newblock {\em Algebr. Geom. Topol.}, 10:1305--1315, 2010.

\bibitem[Goo92]{CalcII}
Thomas~G. Goodwillie.
\newblock Calculus {II}: {A}nalytic functors.
\newblock {\em $K$-Theory}, 5(4):295--332, 1991/92.

\bibitem[GW99]{GW:EI2}
Thomas~G. Goodwillie and Michael Weiss.
\newblock Embeddings from the point of view of immersion theory {II}.
\newblock {\em Geom. Topol.}, 3:103--118 (electronic), 1999.

\bibitem[KS16]{KS:DiagonalComplements}
Sadok Kallel and Ines Saihi.
\newblock Homotopy groups of diagonal complements.
\newblock {\em Algebr. Geom. Topol.}, 16(5):2949--2980, 2016.

\bibitem[LTV10]{LTV:Vass}
Pascal Lambrechts, Victor Turchin, and Ismar Voli\'{c}.
\newblock The rational homology of spaces of long knots in codimension $>2$.
\newblock {\em Geom. Topol.}, 14:2151--2187, 2010.

\bibitem[Mun05]{M:Emb}
Brian~A. Munson.
\newblock Embeddings in the {$3/4$} range.
\newblock {\em Topology}, 44(6):1133--1157, 2005.

\bibitem[Mun08]{M:LinkNumber}
Brian~A. Munson.
\newblock A manifold calculus approach to link maps and the linking number.
\newblock {\em Algebr. Geom. Topol.}, 8(4):2323--2353, 2008.

\bibitem[Mun10]{M:MfldCalc}
Brian~A. Munson.
\newblock Introduction to the manifold calculus of {G}oodwillie-{W}eiss.
\newblock {\em Morfismos}, 14(1):1--50, 2010.

\bibitem[Mun11]{M:Milnor}
Brian~A. Munson.
\newblock Derivatives of the identity and generalizations of {M}ilnor's
  invariants.
\newblock {\em J. Topol.}, 4(2):383--405, 2011.

\bibitem[MV12]{MV:Multi}
Brian~A. Munson and Ismar Voli\'c.
\newblock Multivariable manifold calculus of functors.
\newblock {\em Forum Math.}, 24(5):1023--1066, 2012.

\bibitem[MV14]{MV:Links}
Brian~A. Munson and Ismar Voli{\'c}.
\newblock Cosimplicial models for spaces of links.
\newblock {\em J. Homotopy Relat. Struct.}, 9(2):419--454, 2014.

\bibitem[MV15]{MV:Cubes}
Brian~A. Munson and Ismar Voli{\'c}.
\newblock {\em Cubical homotopy theory}, volume~25 of {\em New Mathematical
  Monographs}.
\newblock Cambridge University Press, Cambridge, 2015.

\bibitem[MW15]{MW:TverbergI}
Isaac Mabillard and Uli Wagner.
\newblock Eliminating higher-multiplicity intersections, {I}. {A} {W}hitney
  trick for {T}verberg-type problems.
\newblock arXiv:1508.02349, 2015.

\bibitem[MW16]{MW:TverbergII}
Isaac Mabillard and Uli Wagner.
\newblock Eliminating higher-multiplicity intersections, {II}. {T}he deleted
  product criterion in the {$r$}-metastable range.
\newblock In {\em 32nd {I}nternational {S}ymposium on {C}omputational
  {G}eometry}, volume~51 of {\em LIPIcs. Leibniz Int. Proc. Inform.}, pages
  Art. 51, 12. Schloss Dagstuhl. Leibniz-Zent. Inform., Wadern, 2016.

\bibitem[Sal02]{S:MultipleImmersions}
Konstantin Salikhov.
\newblock Multiple points of immersions.
\newblock arXiv:math/0203118, 2002.

\bibitem[Sin06]{S:OKS}
Dev~P. Sinha.
\newblock Operads and knot spaces.
\newblock {\em J. Amer. Math. Soc.}, 19(2):461--486 (electronic), 2006.

\bibitem[ST16]{PAST:LinkSSCollapse}
Paul~Arnaud Songhafouo~Tsopm\'en\'e.
\newblock The rational homology of spaces of long links.
\newblock {\em Algebr. Geom. Topol.}, 16(2):757--782, 2016.

\bibitem[Til17]{T:MfldCalcComplex}
Steffen Tillmann.
\newblock Manifold calculus adapted for simplicial complexes.
\newblock arXiv:1702.0560, 2017.

\bibitem[TT13]{T:Context-Free}
Victor Turchin~(Tourtchine).
\newblock Context-free manifold calculus and the {F}ulton-{M}ac{P}herson
  operad.
\newblock {\em Algebr. Geom. Topol.}, 13(3):1243--1271, 2013.

\bibitem[Vol06]{V:FTK}
Ismar Voli{\'c}.
\newblock Finite type knot invariants and the calculus of functors.
\newblock {\em Compos. Math.}, 142(1):222--250, 2006.

\bibitem[Wei99]{W:EI1}
Michael Weiss.
\newblock Embeddings from the point of view of immersion theory {I}.
\newblock {\em Geom. Topol.}, 3:67--101 (electronic), 1999.

\bibitem[Z{\v{Z}}93]{ZZ:Arrangements}
G{\"u}nter~M. Ziegler and Rade~T. {\v{Z}}ivaljevi{\'c}.
\newblock Homotopy types of subspace arrangements via diagrams of spaces.
\newblock {\em Math. Ann.}, 295(3):527--548, 1993.

\end{thebibliography}

\end{document}